\documentclass[12pt,a4paper]{article}

\usepackage{currfile}

\usepackage{theorem,enumerate}
\usepackage{amsmath,latexsym,amssymb,amsfonts, epsfig}
\usepackage{eucal}
\usepackage{mathrsfs}
\usepackage{array}
\usepackage{epstopdf}
\usepackage{MnSymbol}

\newcounter{Scounter}
\theorembodyfont{\normalfont\slshape}
\newtheorem{thm}{Theorem}[section]

\newtheorem{cor}[thm]{Corollary}
\newtheorem{prop}[thm]{Proposition}
\newtheorem{definition}[thm]{Definition}
\newtheorem{lemma}[thm]{Lemma}

\newtheorem{problem}[thm]{Problem}

\newtheorem{rem}[thm]{Remark}

\newtheorem{fact}[thm]{Fact}
\newtheorem{con}[thm]{Conjecture}

\newcommand{\qed}{{$\quad\square$\vs{3.6}}}

\newcommand{\vs}[1]{\vspace*{#1 mm}}

\bfseries\normalfont

\makeatletter
\def\thanks#1{%
   \footnotemark
   \edef\@tempa{\noexpand\noexpand\noexpand\footnotetext[\the\c@footnote]}%
   \toks@\expandafter{\@thanks}%
   \toks\tw@{{#1}}
   \xdef\@thanks{\the\toks@\@tempa\the\toks\tw@}}
\makeatother

\begin{document}

\title{Cycle double covers and non-separating cycles }

\author{
Arthur Hoffmann-Ostenhof, Cun-Quan Zhang, Zhang Zhang}

\date{}
\maketitle

\begin{abstract}
Which $2$-regular subgraph $R$ of a cubic graph $G$ can be extended to a cycle double cover of $G$? We provide
a condition which ensures that every $R$ satisfying this condition is part of a cycle double cover of $G$. As one consequence, we prove that every $2$-connected cubic graph which has a decomposition into a spanning tree and a $2$-regular subgraph $C$ consisting of $k$ circuits with $k\leq 3$, has a cycle double cover containing $C$.
\end{abstract}


\section{Introduction and definitions}

All graphs in this paper are assumed to be finite. A \textit{trivial} component is a component consisting of one single vertex. In context of cycle double covers the following definitions are convenient.
A \textit{circuit} is a $2$-regular connected graph and a \textit{cycle} is a graph such
that every vertex has even degree. Thus every $2$-regular subgraph of a cubic graph is a cycle.\\
In this paper the following concept is essential: a subgraph $C$ of a connected graph $H$ is called \textit{non-separating} if $H-E(C)$ is connected, and \textit{separating} if $H-E(C)$ is disconnected.
Hence, every non-separating cycle $C$ in a connected cubic graph $H$ with $|V(H)|>2$ is an induced subgraph of $H$
if $C$ does not have a trivial component. \\
A \textit{cycle double cover (CDC)} of a graph $G$ is a set $S$ of cycles such that every edge of $G$ is contained in the edge sets of precisely two elements of $S$. The well known \textit{Cycle Double Cover Conjecture (CDCC)} ({\rm \cite{Tutte1987}, \cite{Szekeres1973}, \cite{Itai1978}, \cite{Seymour1979}; or see \cite{Z2})} states that every bridgeless graph has a CDC. 
It is known that the CDCC can be reduced to snarks, i.e cyclically $4$-edge connected cubic
graphs of girth at least $5$ admitting no $3$-edge coloring, see for instance \cite{Z2}. There are several versions of the CDCC, see \cite{Z2}. The subsequent one by Seymour is called the \textit{Strong-CDCC} {\rm (see \cite{Fleischner1984}, \cite{Fleischner1986}, or, see Conjecture 1.5.1 in \cite{Z2})} and it is one of the most active approaches to the CDCC.

\begin{con}
\label{CONJ: strong CDC}
Let $G$ be a bridgeless graph and let $C$ be a circuit of $G$. Then $G$ has a CDC ${\cal S}$ with $C \in {\cal S}$.
\end{con}

Note that the Strong-CDCC can not be modified by replacing ``circuit'' with ``cycle'' since there are
infinitely many snarks which would serve as counterexamples, see \cite{BGHM,HH}. For instance, the Petersen graph $P_{10}$ has a $2$-factor, $C_2$ say, but $P_{10}$ does not have a CDC ${\cal S}$ such that $C_2 \in {\cal S}$. We underline that $C_2$ is separating! Here we only consider CDCs of graphs containing prescribed non-separating cycles.
In particular the following conjecture by the first author has been a motivation for this paper.

\begin{con}\label{con:1}\cite{Open}
Let $C$ be a non-separating cycle of a $2$-edge connected cubic graph $G$. Then $G$ has a CDC $\cal S$ with $C \in \cal S$.
\end{con}

Recall that a \textit{decomposition} of a graph $G$ is a set of edge-disjoint subgraphs covering $E(G)$.
Hence, if a connected cubic graph $G$ has a decomposition into a tree $T$ and a cycle $C$, then $C$ is a non-separating cycle of $G$. Note that all snarks with less than $38$ vertices have a decomposition into a tree and a cycle and that there are infinitely many snarks with such a decomposition, see \cite{HO1}. We consider the following reformulation (see Proposition \ref{p:reformulation}) of the above conjecture.

\begin{con}\label{con:3}
Let $G$ be a $2$-edge connected cubic graph which has a decomposition into a tree $T$ and a cycle $C$.
Then $G$ has a CDC ${\cal S}$ with $C \in {\cal S}$.
\end{con}

Our main result, Theorem \ref{t:tree}, shows that Conjecture \ref{con:3} is true if the cycle $C$ has at most three components. Note that Theorem \ref{t:tree} is valid for all $2$-edge connected graphs. The proof is based on Theorem \ref{t:1} and results which imply the existence of nowhere-zero $4$-flows. Graphs constructed from the Petersen graph demand special treatment in the proof, see Theorem \ref{t:tree} (2). In Section \ref{sep} we consider applications of Theorem \ref{t:1} for cubic graphs, and in Section \ref{open} we present some remarks and one more conjecture.

Note that the tree $T$ in Conjecture \ref{con:3} is a \textit{hist} (see \cite{ABHT}), that is a spanning tree without a vertex of degree two (hist is an abbreviation for homeomorphically irreducible spanning tree). Conversely, every cubic graph with a hist has trivially a decomposition into a tree and a cycle. For informations and examples of snarks with hists, see \cite{HO1,HO1a}. Let us also mention that Conjecture \ref{con:3} limited to snarks is stated in \cite{HO1}.

\begin{prop} \label{p:reformulation}
Conjecture \ref{con:1} and Conjecture \ref{con:3} are equivalent.
\end{prop}

\begin{proof}
Obviously, it suffices to show that the truth of Conjecture \ref{con:3} implies the truth of Conjecture \ref{con:1}. Suppose that $C$ is a non-separating cycle of a $2$-edge connected cubic graph $G$ such that the graph $G_C:=G-E(C)$ is not a tree. Let $T_C$ be a spanning tree of $G_C$. Then the non-trivial components of $G_C-E(T_C)$ can be paths or circuits and all are non-separating in $G_C$. Denote by $X$ the edge set $$\{e \in E(G_C-E(T_C)) \,\,\text{such that e is not contained in a circuit of} \,\, G_C-E(T_C)\}\,\,.$$ Denote by $Y_1$ the $2$-regular subgraph of $G_C-E(T_C)$ which may be empty. Now, subdivide in $G$ each of the edges of $X$ two times and add an edge joining these two new vertices to obtain a circuit of length two and call the union of theses circuits of length two $Y_2$. Thus we obtain a new cubic graph $G'$ and it is straightforward to see that $G'$ has a hist $T'$ such that the $2$-regular subgraph of $G'-E(T')$ denoted by $C'$ consists of $Y_1 \cup Y_2 \cup C$. Obviously every CDC of $G'$ containing $C'$ corresponds to a CDC of $G$ containing $C$. \qed
\end{proof}





For terminology not defined here, we refer to \cite{BM}. For more informations on cycle double covers and flows, see \cite{Z2, Z1}.

\section {Preliminary/lemmas}

The Petersen graph is denoted by $P_{10}$. If $v$ is a vertex of a graph then we denote by $E_v$ the set of edges incident with $v$. A \textit{k-CDC} of a graph $G$ is a set $\cal S$ of $k$ cycles of $G$ such that every edge of $G$ is contained in the edge sets of precisely two elements of $ \cal S$.



\begin{lemma}\label{l:1}
{\rm (Goddyn \cite{Goddyn1988} and Zhang \cite{Zhang1990}, or see \cite{Z1} Lemma 3.5.6)}
Let $G$ be a graph admitting a nowhere-zero $4$-flow and let $C$ be a cycle of $G$. Then
$G$ has a $4$-CDC $\cal S$ with $C \in \cal S$.
\end{lemma}

The following Lemma is well known and can be proven straightforwardly by using a popular result of Tutte, namely that a graph has a nowhere-zero $k$-flow if and only it has a nowhere-zero $\mathbb Z_k$-flow.

\begin{lemma}
\label{LE: 4-f}
Let $G$ be a graph and $C$ be a subgraph
of $G$ such that $G/E(C)$ has a nowhere-zero $k$-flow.
Then $G$ admits a $k$-flow $f$ with $supp(f) \supseteq E(G)-E(C)$.
\end{lemma}

\begin{definition}\label{d:deg}
Let $G$ and $H$ be two graphs. Then $G$ is called {\em $(k,H)$-girth-degenerate} if
$H$ can be obtained from $G$ via a series of contractions of circuits where each has length at most $k$. Moreover, we call $G$ in short $k$-girth-degenerate if $G$ is $(k,K_1)$-girth-degenerate.
\end{definition}

Note that we consider a loop as a circuit of length one. For instance every complete graph is $3$-girth-degenerate and every $2$-connected planar graph is $5$-girth-degenerate. Note also that $H$ in the above definition is a special minor of $G$.

\begin{lemma}\label{LE: Catlin} {\rm (Catlin \cite{Catlin1989}, or, see Lemma 3.8.11 of \cite{Z1}, p. 80)}\\
Let $G$ be a graph and let $C \subseteq G$ be a circuit of length at most $4$.
If $G/E(C )$ admits a nowhere-zero $4$-flow, then so does $G$.
\end{lemma}


\begin{lemma}\label{l:4deg}
Every $4$-girth-degenerate graph $G$ admits a nowhere-zero $4$-flow.
\end{lemma}

\begin{proof}
Apply induction on the number of contractions to obtain $K_1$ (see Definition \ref{d:deg}) and apply Lemma \ref{LE: Catlin}.
 \qed
\end{proof}

\section{Main results}\label{s:2}

Every theorem in this section has been motivated by questions on cubic graphs and was also firstly stated for them. Nevertheless, cubic graphs are not mentioned in the presented theorems since the original results were generalized.

\begin{thm}\label{t:1}
Let $G$ be a $2$-edge connected graph. Suppose that $C$ is a non-separating cycle of $G$ such that $G/E(C)$ has a nowhere-zero $4$-flow. Then $G$ has a $5$-CDC $\cal S$ with $C \in \cal S$.
\end{thm}

\begin{proof}
Since $G/E(C)$ has a nowhere-zero $4$-flow, $G$ has by Lemma \ref{LE: 4-f} a $4$-flow $f$ such that $supp(f) \supseteq E(G)-E(C)$. Set $E_0=\{ e : ~ f(e)=0\}$. Obviously, $E_0 \subseteq E(C)$. Since $G-E(C)$ is connected, there is a circuit $C_e$ of $G-(E(C)-\{ e \})$ containing $e$.
Set $J_1 = \bigtriangleup_{e \in E_0} C_e$. Then $J_1$ contains every edge of $E_0$ but no edge of $C-E_0$.
Moreover, set $J_2 = C \bigtriangleup J_1$. Then $J_2$ is a cycle contained in $supp(f)$ which contains all edges of $C-E_0$. Since $G-E_0$ has a nowhere-zero $4$-flow, there is by Lemma~\ref{l:1} a $4$-CDC ${\cal S}_1$ of $G-E_0$ with $J_2 \in {\cal S}_1$. Then the set ${\cal S} = ({\cal S}_1 - \{ J_2 \})\cup \{ J_1, C \}$ is a $5$-CDC of $G$ with $G\in {\cal S}$. \qed
\end{proof}

Note that Theorem \ref{t:1} and Theorem \ref{t:equiv} below are equivalent statements.
Theorem \ref{t:1} follows from Theorem \ref{t:equiv} since $E_0$ (see the proof of Theorem \ref{t:1}) defines $M$ and thus Theorem \ref{t:equiv} can be applied. The converse direction is shown in the proof of Theorem \ref{t:equiv}. Note also that $E_0$, respectively, $M$ is a matching if $G$ is cubic in Theorem \ref{t:1}, respectively, Theorem \ref{t:equiv}.

\begin{thm}\label{t:equiv}
Let $G$ be a $2$-edge connected graph which contains a non-separating cycle $C$. Suppose that $G$ has an edge subset $M \subseteq E(C)$ satisfying $G-M$ has a nowhere-zero $4$-flow, then $G$ has a $5$-CDC $\cal S$ with $C \in \cal S$.
\end{thm}

\begin{proof}
Since $G-M$ has a nowhere-zero $4$-flow, $G/E(C)$ has a nowhere zero $4$-flow. By applying Theorem \ref{t:1}, the result follows.
\qed
\end{proof}

Note that we can not prove directly Theorem \ref{t:tree} via Theorem \ref{t:1}. Consider for instance the cubic graph, $Q$ say, which results from $P_{10}$ by expanding each $u_1,u_2,u_3$ to a triangle, see Figure 1. Then $Q$ has a decomposition into a tree and a cycle $C$ with three components consisting of triangles. Moreover, $Q/E(C)$ does not have a nowhere-zero $4$-flow and thus Theorem \ref{t:1} can not be applied. Note also that $C$ is not contained in a $5$-CDC of $Q$. We proceed in our preparation for the proof of Theorem \ref{t:tree}.

\begin{prop}\label{p:1,2,3}
Let $G$ be a $2$-edge connected graph with a vertex subset $U$ such that $G-U$ is acyclic. Suppose that $|U| \leq 3$ and that $d_G(v) > 2$ for every $v \in V(G)-U$, then\\
(1) $G$ is $2$-girth-degenerate if $|U|=1$, \\
(2) $G$ is $4$-girth-degenerate if $|U|=2$, and\\
(3) $G$ is $4$-girth-degenerate or $(4,P_{10})$-girth-degenerate if $|U|=3$.
\end{prop}

\begin{proof}
Let $(G, U)$ be a pair such that $G$ is a $2$-edge connected graph and $U \subseteq V(G)$. Suppose that
$(G, U)$ satisfies condition $(\ast)$ which is defined as follows: $G-U$ is acyclic and $d_G(v) > 2$ for every $v \in V(G)-U$.
If there exists a circuit $C \subseteq G$ with $|V(C)| \leq 4$ and $V(C)\cap U \neq \emptyset $, then
we call $C$ a \emph{small circuit} of $(G, U)$ and we set $G':=G/E(C)$ and $U':=\{v_C\} \cup \{U-V(C)\}$ where
$v_C$ is the vertex in $G'$ obtained from contracting $C$. We call $(G', U')$ a \emph{small contraction} of $(G, U)$. Furthermore, we call a sequence of pairs $\{(G_i, U_i)\}_{i=1} ^n$ a \emph{small contraction sequence} if $(G_{i+1}, U_{i+1})$ is a small contraction of $(G_i, U_i)$ for each $i=1,\dots, n-1$. It is clear that if $(G_1, U_1)$ satisfies condition $(\ast)$, then every pair $(G_i, U_i)$ satisfies condition $(\ast)$ for
$i=2,\dots, n$. Note that $|U_1|\geq \cdots \geq |U_n|$ holds and that $U_i$ may equal $V(G_i)$ for some $i$.

For a given $(G,U)$, let $\{(G_i, U_i)\}_{i=1} ^n$ be a maximal small contraction sequence with $(G_1, U_1)=(G,U)$. Hence there is no small circuit
of $(G_n, U_n)$ since the sequence is maximal. In particular, there is no parallel edge with one end in $U$. Denoted by $\hat N_{U_n}(v)$ the neighbors of
$v$ of $G_n$ lying in $U_n$. We say two leaf-vertices $v_1$ and $v_2$ of $V(G_n)-U_n$ are a \emph{bad pair} if
$|\hat N _{U_n}(v_1)\cap \hat N_{U_n}(v_2)|\geq 2$. It is evident that there is no bad pair in $(G_n, U_n)$, otherwise one can easily deduce a small circuit, namely a $4$-circuit by using the bad pair and the two common neighbors of them.
Before we use all of the introduced concepts, we prove the first part of the proposition.\\
{\bf (1}) It suffices to prove that every $(G,U)$ with $U=\{ u \}$ contains a $2$-circuit intersecting $U$ since we then can proceed by induction. Obviously, every component, say $T$, of $G-U$ is a tree. Since $d_G(v) > 2$ (see condition $\ast$) for every leaf-vertex $v \in V(T)$, $v$ is adjacent via an parallel edge to $u$ and thus $G$ has the desired $2$-circuit.  \\
{\bf (2}) To prove Statement (2) we argue by contradiction. So, let $S:=\{(G_i, U_i)\}_{i=1} ^n$ be a maximal small contraction sequence with $(G_1, U_1)=(G,U)$ and suppose that $G_n \not= K_1$.
$G_n-U_n = \emptyset$ would imply that there is a small $2$-circuit which contradicts the maximality of $S$. Thus, there is a component $T$ of $G_n-U_n$. $T$ is not a single vertex otherwise there will be a pair of parallel edges incident with a vertex of $U_n$. Hence $T$ contains two leaf-vertices which form a bad pair, a contradiction. \\
{\bf (3}) Let $S:=\{(G_i, U_i)\}_{i=1} ^n$ be a maximal small contraction sequence with $(G_1, U_1)=(G,U)$ and suppose that $G_n \not= K_1$. We show that $G_n \cong P_{10}$ which will prove Statement (3). Since $|U_1|\geq \cdots \geq |U_n|$ and since Statements (1) and (2) above hold, $|U_n|=3$. Call a vertex subset $W \subseteq V(G_n)-U_n$ a \emph{bad set} if $d_{G_n-U_n}(w_1, w_2)\leq 2$ for any $w_1,w_2\in W$ and $\sum_{w_i\in W} |\hat N_{U_n}(w_i)|\geq 4$. Suppose that $G_n$ has a bad set $W$. The latter inequality implies that one vertex of $U$ has two neighbors in $W$ and the distance condition implies that $G_n$ has a small circuit, a contradiction. Hence, $G_n$ does not have a bad set.\\
Obviously, $G_n-U_n\neq \emptyset$ otherwise $G_n[U_n]$ contains a small circuit. Suppose that $G_n-U_n$ has two components $H_1$ and $H_2$. Recall that $G_n$ does not contain a small circuit. If $H_1$ consists of a single vertex $h_1$, then one can find a vertex $h_2$ in $H_2$ such that $h_1$,$h_2$ form a bad pair. If neither $H_1$ nor $H_2$ is a single vertex, then each contains two leaf-vertices and there are two leaf-vertices of $H_1$ and $H_2$ forming a bad pair by Pigeonhole principle. Hence $G_n-U_n$ is connected and thus a tree which we denote by $T$.\\
$T$ is not a single vertex otherwise there will be a small $3$-circuit. Moreover, $T$ can not have exactly two leaves since then $T$ will be a path $v_0 v_1 \cdots v_k$, and thus either $\{v_0, v_1, v_2\}$ forms a bad set if $k\geq 2$ or $\{v_0, v_1\}$ forms a bad set if $k=1$. Indeed,
$T$ can not have four or more leaves, otherwise one can choose a bad pair from these leaves by Pigeonhole principle. Therefore $T$ has exactly three leaves and thus there is a unique degree $3$-vertex, say $w_0$. Hence $T$ consists of three edge disjoint paths: $w_0 x_1\cdots x_j$, $w_0 y_1\cdots y_k$,
$w_0 z_1\cdots z_l$ with $j,k,l \geq 1$. We claim that $j=k=l=2$. If one of $\{j,k,l\}$, say $j>2$, then $\{x_j,x_{j-1},x_{j-2}\}$ forms a bad set. If one of $\{j,k,l\}$, say $k=1$, then $\{x_1, y_1, z_1\}$ will also form a bad set. Hence $j=k=l=2$. Since there is no bad pair, by symmetry, we may assume that $\hat N_{U_n}(x_2)=\{u_1,u_2\}$, $\hat N_{U_n}(y_2)=\{u_2,u_3\}$, $\hat N_{U_n}(z_2)=\{u_3,u_1\}$ where $U_n=\{u_1,u_2,u_3\}$,
see Figure 1. Since $G_n$ does not have a small circuit, we must also have $x_1 u_3, y_1u_1, z_1u_2\in E(G_n)$. Then $G_n$ is isomorphic to $P_{10}$.\qed
\end{proof}

\begin{figure}[htpb] \label{f:petersen}
\centering\epsfig{file=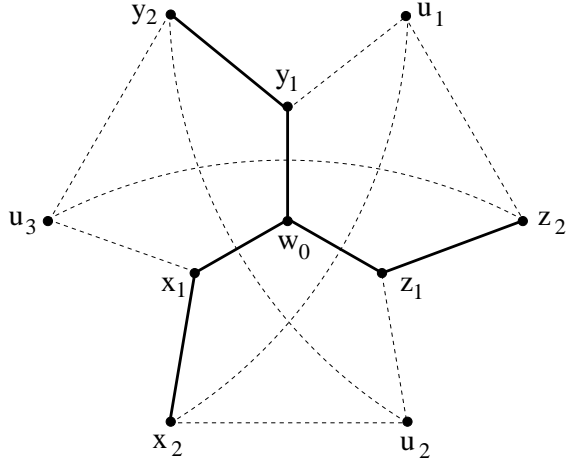,width=2.9in}
\caption{Illustration of $G_n$ and $U_n$ where $G_n \cong P_{10}$ ($T$ is shown in bold face).}
\end{figure}

\begin{lemma}
\label{l:girth-cdc}
Let $G$ be a graph and $C$ a non-separating cycle of $G$.
Let $G/E(C)$ be $(4,H)$-girth-degenerate where $H$ is a graph admitting a $k$-CDC and satisfies $\Delta(H) \leq 3$. Then\\
(1) $G$ has a $(k+1)$-CDC $\mathcal S$ with $C \in \mathcal S$ if $k \geq 5$. \\
(2) $G$ has a $5$-CDC $\mathcal S$
with $C \in \mathcal S$ if $k \leq 4$.
\end{lemma}

\begin{proof}
Let $G_3$ be a $2$-edge connected graph with an edge-cut $E_s$ with $|E_s| = s$, $s \in \{2,3\}$ such that $G_3-E_s$ consists of two components $G_1$ and $G_2$. Define two new graphs $\hat G_1:=G_3/E(G_2)$ and $\hat G_2:= G_3/E(G_1)$.
Denote the unique vertex in $\hat G_1$ ($\hat G_2$) which has been obtained from contracting $E(G_2)$ ($E(G_1)$) by $g_2$ ($g_1$). Let $C_i \subseteq \hat G_i$, $i=1,2$ be a cycle such that $g_1 \notin V(C_2)$ and $g_2 \notin V(C_1)$. Then $C_1$, $C_2$ are cycles of $G_3$ and $C_1 \cup C_2$ is also a cycle of $G_3$. The following fact can be verified straightforwardly.

\begin{fact} \label{f:cdc-cut}
Let $\hat G_i$, $i=1,2$ have a $k_i$-CDC ${\mathcal S}_i$ with $C_i \in {\mathcal S}_i$ and suppose $k_1 \leq k_2$.
Then $G_3$ has a $k_2$-CDC ${\mathcal S}_3$ with $C_1 \cup C_2 \in {\mathcal S}_3$.
\end{fact}

Call a vertex $w_0 \in V(H)$ \textit{big} if it corresponds to a subgraph $W_0$ of $G$ with $|V(W_0)|>1$. Then
$W_0$ is connected and $E_{w_0} \subseteq E(H)$ corresponds to an $s$-edge-cut of $G$ for some $s \in \{2,3\}$ such that one component of $G-E_s$ is $W_0$.\\ 
We prove the lemma by induction on the number of big vertices of $H$ denoted by $b(H)$. If $b(H)=0$, then $G=H$ and $C=\emptyset$ and the lemma holds. Now suppose $b(H)=n+1$. Let $w_0$ be a big vertex of $H$ and let $W_0$ be its corresponding subgraph in $G$. Define the graph $J:=G/E(W_0)$ and the cycle $C_J:=C-(C \cap W_0)$. Then $C_J$ is a non-separating cycle of $J$. Moreover, $J/E(C_J)$ is $(4,H)$-girth-degenerate since $G/E(C)$ is $(4,H)$-girth-degenerate. Furthermore, $H$ has already by assumption a $k$-CDC. Thus all conditions of the considered lemma are fulfilled and since $H$ has (with respect to $J$) precisely $n$ big vertices, $J$ has a CDC ${\mathcal S}_J$ with $C_J \in {\mathcal S}_J$ satisfying Statements (1),(2) (if we replace $G$ by $J$, $S$ by $S_J$, and $C$ by $C_J$).
To obtain the desired CDC of $G$, define the graph $J':=G/E(G-V(W_0))$ (recall that $W_0$ is connected) and denote the unique vertex of $J'$ which is not part of $W_0$ by $x$. Set $C_{J'}=C-C_J$. Since $G/E(C)$ is $(4,H)$-girth-degenerate and $w_0$ a big vertex, it follows that $J'/E(C_{J'})$ is $4$-girth-degenerate. Hence $J'/E(C_{J'})$ has by Lemma \ref{LE: Catlin} a nowhere-zero $4$-flow. Since $C_{J'}$ is a non-separating cycle of $J'$, there is by Theorem \ref{t:1} a $5$-CDC ${\mathcal S}_{J'}$ of $J'$ with $C_{J'} \in {\mathcal S}_{J'}$. Depending on the value of $k$ (concerning the $k$-CDC of $H$) there are two cases.\\
Case 1. $k \geq 5$. Then ${\mathcal S}_{J}$ is a $(k+1)$-CDC of $J$. Since ${\mathcal S}_{J'}$ is a $5$-CDC of $J'$, and $k+1 > 5$, Fact \ref{f:cdc-cut} implies that $C= C_{J'} \cup C_J$ is contained in a $(k+1)$-CDC ${\mathcal S}$
of $G$ (note that $x \notin V(C_{J'})$ and that $w_0 \notin V(C_J)$).\\
Case 2. $k \leq 4$. Then ${\mathcal S}_{J}$ is a $5$-CDC of $J$ and ${\mathcal S}_{J'}$ is a $5$-CDC of $J'$.
Fact \ref{f:cdc-cut} implies that $C= C_{J'} \cup C_J$ is contained in a $5$-CDC ${\mathcal S}$
of $G$. \qed
\end{proof}

\newpage

\begin{thm}\label{t:tree}
Let $G$ be a $2$-edge connected graph with a decomposition into a tree $T$ and a cycle $C$ with $k \leq 3$ components. Then $G$ has a CDC $\cal S$ with $C \in \cal S$ and in particular the following holds. \\
(1) If $k \leq 2$, then $G$ has a $5$-CDC ${\cal S}_2$ with $C \in {\cal S}_2$. \\
(2) Let $k =3$. Then $G$ has a $5$-CDC ${\cal S}_3$ with $C \in {\cal S}_3$ if
$G$ is not contractible to the Petersen graph, otherwise $G$ has a $6$-CDC ${\cal S}_3'$ with $C \in {\cal S}_3'$.
\end{thm}

\begin{proof}
It is straightforward to see that we can assume that $G$ does not have a vertex of degree two.
Moreover, we can also assume that $V(C) \subseteq V(T)$. To see this, we introduce the following definition. Let $H$ be a graph and $v \in V(H)$ with $av,bv \in E_v$. Then we say that the graph $(H-av-bv) \cup ab $ is obtained from $H$ by \textit{splitting away} the edges $av$ and $bv$. If $V(C) \not\subseteq V(T)$, we form from $G$ and $C$ a new graph $\hat G$ (without changing the tree $T$) and a new cycle $\hat C \subseteq \hat G$ (having again $k$ components). Regard each component $C_i$, $i \in \{1,...,k\}$ of $C$ as an eulerian closed trail. For every vertex $v \in V(C_i)$ in $G$ with $d_{C_i}(v) \geq 4$ satisfying $v \notin V(T)$, we split repeatedly pairs of consecutive edges (of the trail) having both $v$ as endvertex, away, until $T$ becomes a spanning tree and we denote this obtained cycle by $\hat C$. It is straightforward to verify that every $r$-CDC of $\hat G$ with contains $\hat C$ corresponds to a $r$-CDC of $G$ with contains $C$.
Hence we assume $V(C) \subseteq V(T)$. \\ 
Since $C$ has at most three components, $G'=G/E(C)$ satisfies the conditions of Proposition \ref{p:1,2,3} (replace $G$ by $G'$). We can assume that $G'$ is not $(4,P_{10})$-girth-degenerate, otherwise we apply Proposition \ref{p:1,2,3} (3) and Lemma \ref{l:girth-cdc} (since $P_{10}$ has a $5$-CDC). Thus $G'$ is at most $4$-girth-degenerate by Proposition \ref{p:1,2,3}. By Lemma \ref{l:4deg}, $G'$ admits a nowhere-zero $4$-flow. Moreover, $C$ is non-separating since $G-E(C)$ is a tree. Hence the conditions of Theorem \ref{t:1} are fulfilled and its application finishes the proof.
\qed
\end{proof}

\section{
Corollaries for cubic graphs
}\label{sep}

Within this section we show some applications of Theorem \ref{t:1} and Theorem \ref{t:tree} 
for cubic graphs. For this purpose we need the following definition and lemma.

\begin{definition}
\label{DEF: 1}
An evenly spanning cycle of a graph $G$ is a spanning cycle $C$ of $G$ such that for every component, $L$ say, of $C$, the number of vertices in $L$ with odd degree (with respect to $G$) is even.
\end{definition}

For instance, $V(G)$ is an evenly spanning cycle of $G$ if $G$ is an eulerian graph. In contrast to the latter example, an evenly spanning cycle of a $2k+1$-regular graph can not contain a trivial component. Note that every hamiltonian circuit is an evenly spanning cycle.

\begin{lemma}\label{l:2}
{(\rm \cite{Z2} or \cite{Archdeacon1984})}
 The following statements are equivalent:\\
(1) A graph $G$ has a nowhere-zero $4$-flow. (2) $G$ has an evenly spanning cycle.
\end{lemma}

\begin{cor}
\label{c:ham}
Let $G$ be a $2$-edge connected cubic graph. Suppose that $C$ is a non-separating cycle of $G$ such that $G/E(C)$ has a hamiltonian circuit. Then $G$ has a $5$-CDC $\cal S$ with $C \in \cal S$.
\end{cor}

\begin{proof}
Since a hamiltonian circuit in $G':=G/E(C)$ is an evenly spanning cycle, $G'$ has by Lemma \ref{l:2} a nowhere-zero $4$-flow. By applying Theorem \ref{t:1}, the result follows.
\qed
\end{proof}

\begin{cor}
\label{c:3}
Let $G$ be a $2$-edge connected cubic graph with a $2$-factor consisting of two chordless circuits $C_1$, $C_2$.
Then $G$ has a $5$-CDC $\mathcal S$ with $C_1 \in \mathcal S$.
\end{cor}

\begin{proof}
Since $C_1$ is non-separating and $G/E(C_1)$ is hamiltonian, the result follows by applying Corollary \ref{c:ham}.\qed
\end{proof}

\begin{rem}
$C_1$ in Corollary \ref{c:3} is part of some CDC even if $C_1$ is allowed to have chords, see \cite{FH}.
$G$ in Corollary \ref{c:ham} has some CDC even if $C$ is separating, see \cite{HM1}.
The above results offer some insight which cycles are part of a $5$-CDC, see the Strong $5$-CDCC in \cite{HO2}.
\end{rem}

The next result follows directly from Theorem \ref{t:tree}.


\begin{cor}\label{c:twotrees}
Let $G$ be a $2$-edge connected cubic graph with a cycle $C \subseteq G$ such that \\(i) $C$ has at
most three components and \\(ii) $G - E(C)$ is acyclic and has at most two components $\{T_1, T_2\}$.
\\Then $G$ has a CDC if $T_k \cup C$ is bridgeless for each $k \in \{1,2\}$.
\end{cor}

\begin{cor}\label{c:decomp}
Let $G$ be a $2$-edge connected cubic graph which has a decomposition into a spanning tree $T$, $k_1$ circuits and
$k_2$ edges such that $k_1 + k_2 \leq 3$. Then $G$ has a CDC containing the cycle consisting of the $k_1$ circuits.
\end{cor}

\begin{proof}
Since the CDCC is known to hold for graphs with small order, we can assume that $k_1 \not=0$.
Subdivide each of the $k_2$ edges two times and add an edge joining these two vertices to obtain a circuit of length two. Then we obtain a new graph $G'$ with a hist $T'$ for which we can apply Theorem \ref{t:tree} since $G'-E(T')$ has $k_1 + k_2 \leq 3$ circuits. Moreover, the CDC of $G'$ corresponds to a CDC of $G$ which contains all $k_1$ circuits of $G-E(T)$. \qed
\end{proof}

\begin{cor}\label{c:cy4}
Every cyclically $4$-edge connected cubic graph which has a decomposition into a tree and a cycle $C$ consisting of $k$ circuits with $k \leq 3$ has a $5$-CDC $\mathcal S$ with $C \in \mathcal S$.
\end{cor}

\begin{proof}
Every cubic graph which is contractible to $P_{10}$ is either $P_{10}$ itself or a cubic graph with a cyclic $3$-edge cut. Since for every decomposition of $P_{10}$ into a tree and a $2$-regular subgraph, the $2$-regular subgraph consists of one circuit, the proof follows by applying Theorem \ref{t:tree}. \qed
\end{proof}

\section{Remarks and open problems}\label{open}


We know that Conjecture \ref{con:1} is not implied by Theorem \ref{t:1} (recall the graph $Q$ defined below the proof of Theorem \ref{t:equiv}). Is this still the case if we restrict Conjecture \ref{con:1} to snarks? The graph $Q^*$ illustrated in Figure 2 is a snark which has a non-separating cycle $C^*$ (which is contained in a CDC) but Theorem \ref{t:1} is not applicable since $Q^*/E(C^*)$ does not have a nowhere-zero $4$-flow. $Q^*$ is constructed from the graph $P'$ which admits no nowhere-zero $4$-flow, by contracting double edges and expanding vertices of degree five to $5$-circuits, see Fig. 12.1 on page 303 in \cite{K}. Observe also that $C^*$ is a maximal non-separating cycle of $Q^*$, i.e $Q^*$ does not have a larger non-separating cycle $\hat C$ satisfying $C^* \subset \hat C$.\\
With respect to Conjecture \ref{con:3}, we do not know a cyclically $4$-edge connected cubic graph which prevents the direct application of Theorem \ref{t:1}.

\begin{problem}
Does there exist a snark $G$ which has a decomposition into a tree and a cycle $C$ such that $G/E(C)$ does not have a nowhere-zero $4$-flow?
\end{problem}

\begin{figure}[htpb] \label{f:kochol}
\centering\epsfig{file=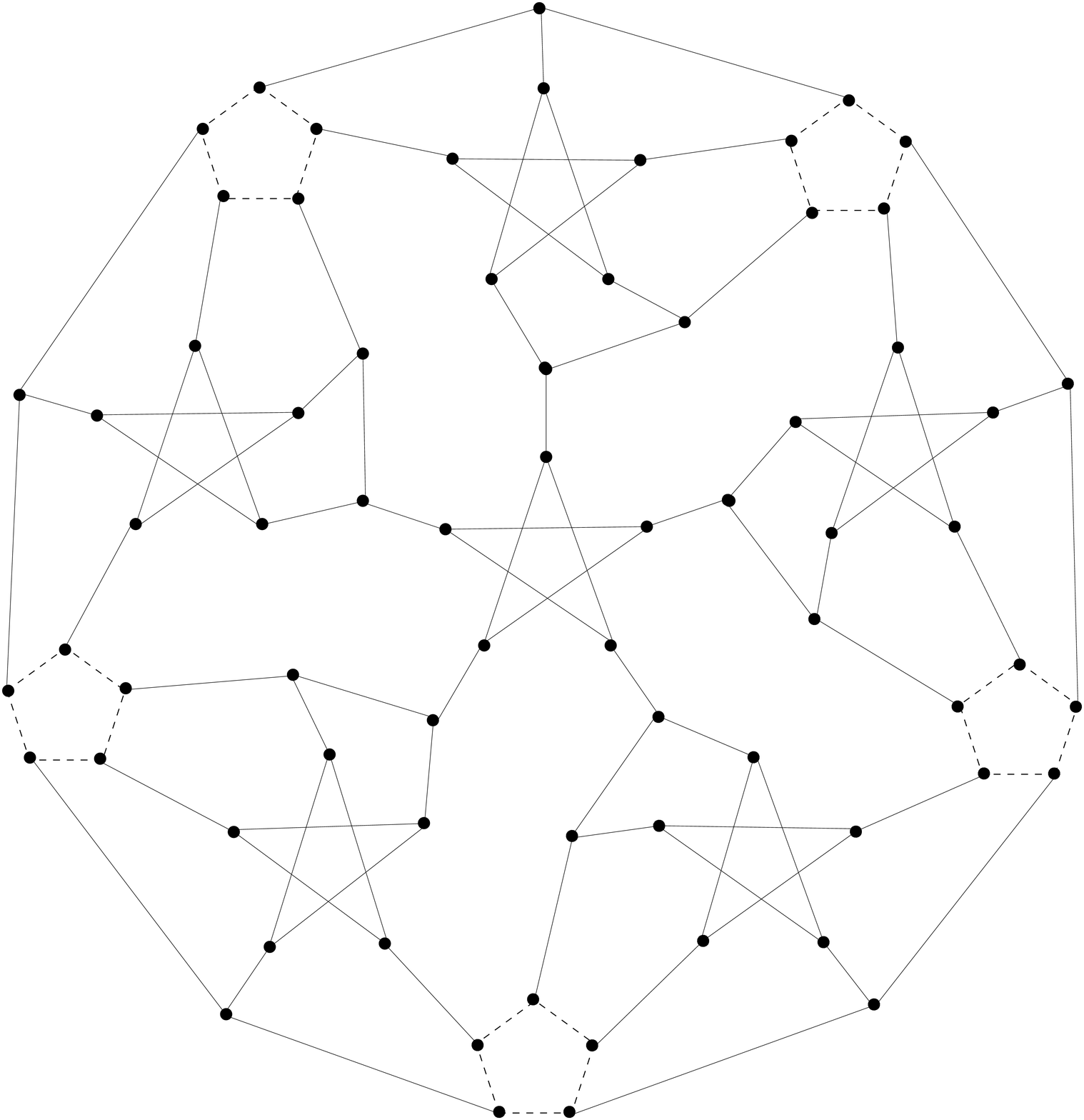,width=4.8in}
\caption{A snark $Q^*$ with a non-separating cycle $C^*$ illustrated by dashed edges.}
\end{figure}

The truth of the next conjecture implies the truth of the CDCC and in particular the truth of the $5$-CDCC, see Theorem \ref{t:1}.

\begin{con}\label{con:4}
Every cyclically $4$-edge connected cubic graph $G$ contains a non-separating cycle $C$ such that $G/E(C)$ has a nowhere-zero $4$-flow.
\end{con}

Note that Conjecture \ref{con:4} would be false if $G$ is not demanded to be cyclically $4$-edge connected. For instance, the cyclically $3$-edge connected cubic graph which is obtained from $K_4$ by replacing every vertex of $K_4$ with a copy of $P_{10}-v, \,v\in V(P_{10})$ would then form a counterexample.



\section*{Acknowledgments}
A.~Hoffmann-Ostenhof was supported by the Austrian Science Fund (FWF) project P 26686.
C.-Q.~Zhang was supported by National Security Agency, No.~H98230-16-1-0004, and National Science Foundation, No.~DMS-1700218.



\begin{thebibliography}{1}

\bibitem{Archdeacon1984}
D.Archdeacon, Face coloring of embedded graphs.
J. Graph Theory 8 (1984) 387-398.


\bibitem {BM} J.A.Bondy, U.S.R. Murty,
\newblock Graph Theory, Springer (2008).

 	
\bibitem {ABHT} M.O.Albertson, D.M.Berman, J.P.Hutchinson, C.Thomassen, Graphs with homeomorphically irreducible spanning trees.
\newblock Journal of Graph Theory 14 (2) (1990) 247-258.


\bibitem {BGHM} G. Brinkmann, J. Goedgebeur, J. H\"agglund, K. Markstr\"om, Generation and properties of snarks.
\newblock Journal of Combinatorial Theory, Series B 103 (2013) 468-488.

\bibitem{Catlin1989} P.A.Catlin, Double cycle covers and the Petersen graph.\\ Journal of Graph Theory 13 (1989) 465-483.

\bibitem{Fleischner1984}
H.Fleischner, 1984. Cycle decompositions, 2-coverings, removable cycles
and the four-color disease. Pages 233-246 of: Bondy, J. A., and Murty,
U. S. R. (eds), Progress in Graph Theory. New York: Academic Press.

\bibitem{Fleischner1986}
H.Fleischner, Proof of the strong 2-cover conjecture for planar
graphs. \\Journal of Combinatorial Theory, Series B 40 (2) (1986) 229-230.


\bibitem {FH} H.Fleischner, R.H\"aggkvist,\\ Cycle double covers in cubic graphs having special structures. \\
Journal of Graph Theory 77 (2) (2014) 158-170.


\bibitem{Goddyn1988} L.A.Goddyn, Cycle covers of graphs, Ph.D. Thesis, University of Waterloo, 1988.

\bibitem {HM1} R.H\"aggkvist, K.Markstr\"om, Cycle double covers and spanning minors I.\\
Journal of Combinatorial Theory, Series B 96 (2) (2006) 183-206.

\bibitem {HH} J.H\"agglund, A.Hoffmann-Ostenhof, Construction of permutation snarks.\\
Journal of Combinatorial Theory, Series B 122 (2017) 55-67.

\bibitem {HO1} A.Hoffmann-Ostenhof, T.Jatschka,
\newblock Snarks with special spanning trees.\\ arXiv:1706.05595 (2017).

\bibitem {HO1a} A.Hoffmann-Ostenhof, T.Jatschka,
\newblock Special Hist-Snarks. arXiv:1710.05663 (2017).


\bibitem {HO2} A.Hoffmann-Ostenhof, A Note on 5-Cycle Double Covers.\\
\newblock Graphs and Combinatorics 29 (4) (2013), 977-979.








\bibitem{Itai1978}
A.Itai, M.Rodeh, 1978. Covering a graph by circuits. Pages
289-299 of: Automata, Languages and Programming. Lecture Notes in
Computer Science, vol. 62. Berlin: Springer-Verlag.


\bibitem {K} M.Kochol, Superposition and Constructions of Graphs Without Nowhere-zero k-flows.\\
European Journal of Combinatorics 23 (2002) 281-306.

\bibitem {Open} 
www.openproblemgarden.org/op/cycle\_double\_covers\_containing\_pre \\ defined\_2\_regular\_subgraphs, 2017.

\bibitem{Seymour1979}
P.D.Seymour, 1979. Sums of circuits. Pages 342-355 of: Bondy, J.A.,
and Murty, U.S.R. (eds), Graph Theory and Related Topics. New York:
Academic Press.

\bibitem{Szekeres1973}
G.Szekeres, Polyhedral decompositions of cubic graphs.\\ Bull.
Austral. Math. Soc. 8 (1973) 367-387.



\bibitem{Tutte1987}
W.T.Tutte, personal correspondence with H. Fleischner (July
22, 1987).








\bibitem{Zhang1990} C.-Q. Zhang, Minimum cycle coverings and integer flows. \\Journal of Graph Theory, 14 (1990) 537-546.


\bibitem {Z1} C.-Q.Zhang, Integer flows and cycle covers of graphs. CRC Press (1997).

\bibitem {Z2} C.-Q.Zhang, Circuit double cover of graphs. Cambridge University Press (2012).









\end{thebibliography}
\end{document}